\documentclass[12pt]{article}
\usepackage{makeidx}
\usepackage{amssymb}
\usepackage{amsmath}
\usepackage{amsthm}
\usepackage{graphicx}
\usepackage{subfigure}
\usepackage{epsf}
\input cyracc.def
\newfam\cyrfam
\font\tencyr=wncyr10
\font\sevencyr=wncyr7
\font\fivecyr=wncyr5
\def\cyr{\fam\cyrfam\tencyr\cyracc}
\textfont\cyrfam=\tencyr \scriptfont\cyrfam=\sevencyr%
\scriptscriptfont\cyrfam=\fivecyr

\newtheorem{theorem}{Theorem}

\newtheorem{proposition}[theorem]{Proposition}

\theoremstyle{definition}
\newtheorem*{remark}{Remark}

\newtheorem*{heuristic1}{Heuristic for $M_k(T)$}
\newtheorem*{heuristic2}{Heuristic for $M'_k(T)$}

\newcommand{\ttt}{\scriptscriptstyle}
\newcommand{\Q}{\mathbb Q}

\newcommand{\C}{\mathbb C}
\newcommand{\Z}{\mathbb Z}
\newcommand{\Hy}{\mathbb H}
\newcommand{\F}{\mathbb F}

\newcommand{\Cl}{\mathrm{Cl}}

\newcommand{\TSg}{\mbox{{\cyr X}}}

\def\TSg{\mbox{{\cyr X}}}

\def\({\left(}
\def\){\right)}

\DeclareMathOperator{\vol}{vol}

\DeclareMathOperator{\lcm}{lcm}

\def\squareforqed{\hbox{\rlap{$\sqcap$}$\sqcup$}}\def\qed{\ifmmode\squareforqed\else{\unskip\nobreak\hfil
\penalty50\hskip1em\null\nobreak\hfil\squareforqed
\parfillskip=0pt\finalhyphendemerits=0\endgraf}\fi}
\begin{document}

%\frontmatter
\pagestyle{headings}
%\vspace{1cm}

\title{Regulators of rank one quadratic twists}

%first author

\author{Christophe DELAUNAY \\[5pt]
  Xavier-Fran\c{c}ois ROBLOT \\[3pt]
  \small Universit\'e de Lyon,
  Universit\'e Lyon 1,\\
  \small CNRS, UMR 5208 Institut Camille Jordan,\\
  \small Batiment du Doyen Jean Braconnier,\\
  \small 43, blvd du 11 novembre 1918,\\
  \small F - 69622 Villeurbanne Cedex, France \\[3pt]
  \small email: \texttt{\{delaunay,roblot\}@math.univ-lyon1.fr} \\
  \small url:
  \texttt{http://math.univ-lyon1.fr/$\sim$\{delaunay,roblot\}} \\[5pt]
}

\maketitle

\begin{abstract}
  We investigate the regulators of elliptic curves with rank 1 in some
  families of quadratic twists of a fixed elliptic curve. In
  particular, we formulate some conjectures on the average size of
  these regulators. We also describe an efficient algorithm to compute
  explicitly some of the invariants of a rank one quadratic twist of
  an elliptic curve (regulator, order of the Tate-Shafarevich group,
  etc.) and we discuss the numerical data that we obtain and compare
  it with our predictions.
\end{abstract}

\bigskip

\section{Introduction and notations}

We study the regulators of elliptic curves of rank 1 in a family of
quadratic twists of a fixed elliptic curve $E$ defined over
$\Q$. Methods coming from Random Matrix Theory, as developed in
\cite{keating-snaith}, \cite{CKRS}, \cite{CFKRS}, etc., allow us to
derive precise conjectures for the moments of those regulators. Our
hope is that these moments will help to make predictions for the
number of curves with extra-rank (i.e. the number of even quadratic
twists\footnote{An odd (resp. even) quadratic twist of $E$ is a
  quadratic twist such that the sign of the functional equation of its
  $L$-function is $-1$ (resp. $+1$). By the Birch and Swinnerton-Dyer
  conjecture this is equivalent to say that its Mordell-Weil rank is
  odd (resp. even)} with a Mordell-Weil rank $\geq 2$, or the number
of odd quadratic twists with Mordell-Weil rank $\geq 3$).  Then, we
describe an efficient method, using Heegner-point construction, for
computing the regulator (and the order of the Tate-Shafarevich group)
of an elliptic curve of rank 1 in a family of quadratic
twists. Finally, we discuss and compare our extensive numerical data
(for some families of odd quadratic twists of the curves $11a1$,
$14a1$, $15a1$ and $17a1$) with our predictions.

\medskip 

From a numerical and experimental point of view, the situation of odd
quadratic twists really differs from the one of even quadratic
twists. Indeed, in the latter case, for each curve $E_d$ in a family
$(E_d)_d$ of even quadratic twists of a fixed elliptic curve $E$, one
has to compute the special value $L(E_d,1)$ of its $L$-function at
$s=1$ and determine if it is zero or not. If $L(E_d,1)=0$ then the
curve $E_d$ has extra-rank. Otherwise the curve has
rank $0$, the regulator is simply $1$, and the Birch and
Swinnerton-Dyer conjecture allows us to deduce the value of $|\TSg(E_d)|$
from that of $L(E_d,1)$. The computation of $L(E_d,1)$ is done via a
Waldspurger's formula which, roughly speaking, states that $L(E_d,1)$
is, up to a fudge factor, the square of the $|d|$-th coefficient of a
weight $3/2$ modular form given by an explicit linear combination of
theta series. It follows that, in this case, computations are possible
for very large families of quadratic twists (see for example
\cite{rubinstein}, \cite{quattrini}, etc.). Note that the numerical
data coming from these computations are in close agreement with the
well-known conjectures of \cite{CKRS} about extra-vanishing (coming
from the models of Random Matrix Theory), or on the behavior of the
Tate-Shafarevich groups $\TSg(E_d)$ of $E_d$ (see \cite{quattrini},
\cite{delaunay}).

\medskip 

In the rank 1 case, numerical investigation appears to be much more
complicated and, as far as we know, has never been done before. In
that case, we first have to compute the value of the derivative
$L'(E_d,1)$ for each curve $E_d$ in the family of odd quadratic
twists. However, there is no Waldspurger's formula to compute this
value directly, and furthermore from this value one can only deduce
(assuming it is non-zero and under the Birch and Swinnerton-Dyer
conjecture) the value of the product $R(E_d) \, |\TSg(E_d)|$ where
$R(E_d)$ is the regulator of $E_d$. Thus we also need to be able to
evaluate at least one of the two terms of this product.\footnote{For
  some families of elliptic curves $(F_j)_j$, there exists a generic
  point in the Mordell-Weil group $F_j(\Q)$, thus one can separate the
  terms in this product and a direct investigation is possible (see
  \cite{del-duq}).  However, such families for which we know in
  advance the regulator are very special and in particular are not
  quadratic families, although we must say that it is possible to get
  sometimes a generic point for some very specific and tiny sub-family
  of quadratic twists.} The only (known) efficient way to do
this is to write down a generator $G_d$ of $E_d(\Q)$ and to compute
$R(E_d) = \hat{h}(G_d)$ where $\hat{h}$ is the canonical
height\footnote{This equality fixes once and for all our choice of the
  canonical height.  Note that this height is {\em twice} the height
  in Silverman's book \cite{silverman} or in Krir's paper \cite{krir}
  so this explains the difference of a factor 2 between the formulae
  in this paper and theirs.} of $E_d$.

\medskip

The method we used in this paper is to first adapt the Heegner-point
construction to our situation in order to construct a generator $G_d$
and then replace the Waldspurger's formula by the formula of Gross and
Zagier. This allows us to compute directly the regulator $R(E_d)$ {\em
  and at the same time} the order of the Tate-Shafarevich group
$|\TSg(E_d)|$ (assuming the Birch and Swinnerton-Dyer conjecture).

\bigskip

\noindent\textbf{Hypothesis. From now on, we assume the truth of the
Birch and Swinnerton-Dyer conjecture.}

\bigskip 

We now give some notations. Fix an elliptic curve $E$ defined over
$\Q$ and let $N$ be its conductor. The $L$-function of $E$ is
$$
L(E,s)=\sum_{n\geq 1} a(n)n^{-s}\;,\;\; \Re(s)>3/2
$$
It is now a classical and deep result that $L(E,s)$ can be
analytically continued to the whole complex plane and satisfies a
functional equation:
$$
\Lambda(E,s):= \left(\frac{\sqrt{N}}{2\pi}\right)^s \Gamma(s)L(E,s)=w
\Lambda(E,2-s)
$$ 
where $w=\pm 1$ gives the parity of the order of vanishing of $L(E,s)$
at $s=1$. Let $d$ be a fundamental discriminant. We denote by $E_d$
the quadratic twist of $E$ by $d$. The curves $E$ and $E_d$ are
isomorphic over the quadratic field $\Q(\sqrt{d})$ but not over $\Q$.
We denote by $\psi_d$ ($\psi$ if $d$ is clear in the context) the
isomorphism between $E$ and $E_d$ defined in the following way. Assume
that the curves $E$ and $E_d$ are given by:
\begin{eqnarray*}
E \; &:& \; y^2=x^3+Ax^2+Bx+c \\
E_d \; &:& \; y^2=x^3+Adx+Bd^2x+Cd^3
\end{eqnarray*}
then $\psi_d$ is:
$$
\begin{array}{cccc}
\psi_d \; \; : & \; \; E & \stackrel{\sim}{\longrightarrow} & E_d \\
             & (x,y)   & \longmapsto                      & (dx,d^{3/2}y)    
\end{array}
$$
The non-trivial automorphism $x \mapsto \bar{x}$ of $\Q(\sqrt{d})$,
which is the restriction of the complex conjugation if $d < 0$, acts
by:
\begin{equation}\label{conj_de_psi}
\psi_d(\overline{P})=-\overline{\psi_d(P)}
\end{equation}

Whenever $d$ and $N$ are coprime (and this will always be the case in
our families), the conductor of $E_d$ is $Nd^2$ and we have:
$$
L(E_d,s)= \sum_{n\geq 1} a(n)\chi_d(n) n^{-s}
$$ 
where $\chi_d(.)=\left(\frac{d}{.}\right)$ is the quadratic character
associated to $d$. The sign of the
functional equation satisfied by $L(E_d,s)$ is 
$$
w(E_d)=w \cdot \chi_d(-N).
$$
In the odd rank case (i.e. $w(E_d)=-1$), we are interested in the
values at $s=1$ of the derivatives of the $L$-functions. We have:
$$
L'(E_d,1)=\frac{\Omega(E_d) \, c(E_d)}{|E_d(\Q)\mbox{{\scriptsize
      tors}}|^2} \, R(E_d) \, S(E_d) 
$$
where as usual $\Omega(E_d)$ is the real period, $R(E_d)$ is the
regulator and $c(E_d) =\prod_{p\mid Nd} c_p(E)$ is the product of the
local Tamagawa numbers. The Birch and Swinnerton-Dyer conjecture
predicts that $S(E_d)=|\TSg(E_d)|$ if $L'(E_d,1) \neq 0$ and
$S(E_d)=0$ otherwise.

\section{Families of quadratic twists}

For each prime $p$ dividing the conductor $N$ of $E$, we fix a sign
$w_p=\pm 1$ so that $\prod_{p\mid N} w_p=w$. We then define the set:
$$
{\mathcal F}=\Big\{d<0, \mbox{ fundamental discriminant with}
\left(\frac{d}{p}\right)=w_p \mbox{ for all } p\mid N\Big\}
$$
and we let:
$$
{\mathcal F}(T)=\Big\{ d \in {\mathcal F},\; |d|<T \Big\}
$$
Then, our family of quadratic twists is the set $\left(E_d\right)_{d
  \in {\mathcal F}}$ and, for all these curves $E_d$, we have
$w(E_d)=-1$ by the above assumption on the product of the $w_p$'s. It
will be convenient for us to partition the family ${\mathcal F}$ into two
subfamilies corresponding to the odd and even discriminant
cases. Therefore we define:
$$
{\mathcal F}_{\mbox{{\scriptsize odd}}}= \Big\{ d \in {\mathcal F},\;
d \mbox{ odd} \Big\} 
\quad\mbox{and}\quad
{\mathcal F}_{\mbox{{\scriptsize odd}}}(T)= \Big\{ d \in {\mathcal
  F(T)},\; d \mbox{ odd} \Big\} 
$$
Note that we will not need to consider the subfamilies corresponding
to the even discriminants.

For $d \in {\mathcal F}$ with $|d|$ large enough, it follows from
Proposition 2 of \cite{delaunay2} that, if we denote by $c_4$ the
usual invariant of $E$ (cf. \cite[\S 7.1]{cohen0}), we have:
\begin{equation}\label{SR}
  S(E_d) \, R(E_d) = \frac{\sqrt{|d|} \, L'(E_d,1)}
  {\delta_8(d,c_4) \, \Omega_{\mathcal F} \, \prod\limits_{p\mid d}
    c_p(E_d)}   
\end{equation}
where $\delta_8(d,c_4)=2$ if $8 \mid d$ and $2 \mid c_4$, and
$\delta_8(d,c_4)=1$ otherwise, and $\Omega_{\mathcal F}$ is some
positive number which does not depend on $d$. When $L'(E_d,1)$ is not
zero then $E_d(\Q)$ has rank 1 and the regulator $R(E_d)$ is equal to
the canonical height $\hat{h}(G_d)$ of a generator $G_d$ of $E_d$. So,
the problem of studying the behavior of $R(E_d)$ is roughly speaking
the same as the one of studying the complexity of rational solutions
of the associated Diophantine equations.

\subsection{On upper bounds for $h(G_d)$}

Lang's conjecture \cite[Conjecture 10.2]{silverman} predicts that for
a general elliptic curve $E$:
$$
R(E) \ll |\Delta_{\mbox{{\scriptsize min}}}(E)|^{1/2+\epsilon}
$$ 
where $\Delta_{\mbox{{\scriptsize min}}}(E)$ is the minimal
discriminant of $E$. In our family, we have
$\Delta_{\mbox{{\scriptsize min}}}(E_d)=d^6\Delta_{\mbox{{\scriptsize
      min}}}(E)$ hence, this yields:
$$
R(E_d) \ll |d|^{3+\epsilon}
$$
Of course, this upper bound is very far from what we really expect for our family.
Indeed, using equation~(\ref{SR}) and the fact that $S(E_d)$ and
$c_p(E_d)$ are positive integers (so greater or equal to 1), the
Lindel\"of hypothesis applied to $L'(E_d,1)$ gives the following
conditional upper bound:
$$
R(E_d) \ll_\epsilon |d|^{1/2+\epsilon}
$$

In some cases, this upper bound can be proved on average. Anticipating
on the results and notations of Section~\ref{algo}, we prove:
\begin{proposition}\label{apres} 
  Assume that $N$ is square-free, $L(E,1)\neq 0$ and $w_p=+1$ for all
  $p \mid N$. Then we have:
  \begin{equation}\label{upper}
    \frac{1}{|{\mathcal F}_{\ttt odd}(T)|} \sum_{{\begin{tabular}{c}\\[-6mm] $\scriptscriptstyle d\in{\mathcal F}_{\ttt odd}(T) $ \\ [-2mm]$\scriptscriptstyle L^\prime(E_d,1)\neq 0$ \end{tabular}}}\hspace*{-4mm} R(E_d) \ll T^{1/2}\log T
  \end{equation}
\end{proposition}
\begin{proof}
  This is a direct corollary of a theorem of Ricotta and Vidick.
  Indeed, with the notations of section \ref{algo} we have
  $R(E_d)=\hat{h}(G_d)\leq \hat{h}(R_d)=4\hat{h}_E(P_d)$, where
  $\hat{h}_E$ is the canonical height on $E$ and $P_d \in E(\Q{\sqrt{d}})$
  is the Heegner point constructed in \ref{algo}. Now, we apply the
  corollaire 3.2 of \cite{ricotta-vidick}. 
\end{proof}

\begin{remark}
  Classical conjectures predict that the number of discriminants $d$
  in our family for which $L'(E_d,1)=0$ should have density $0$ (we
  will come back to this fact later), so $|{\mathcal F}_{\ttt
    odd}(T)|$ is roughly the number of terms in the sum of the formula
  above and hence the proposition really asserts that on average
  $R(E_d) \ll |d|^{1/2+\epsilon}$ for all $d\in {\mathcal F}_{\ttt
    odd}$.
\end{remark}

\subsection{On lower bound for $R(E_d)$}

Another conjecture of Lang asserts that $\hat{h}(G_d) \gg \log
|\Delta_{\mbox{{\scriptsize min}}}(E_d)|$, thus we get:
\begin{equation}\label{lang}
\hat{h}(G_d) \gg \log |d|
\end{equation}
In fact, we have the more precise result:
\begin{proposition}\label{explicite} 
  If $j(E)\neq 0$, $1728$, then there is an explicit constant $M$,
  depending on $E$ and on the $w_p$, such that we have for all $d\in
  {\mathcal F}$:
  $$
  \hat{h}(G_d) > \frac{1}{M} \log |d|
  $$
  If $w_p=+1$ for all $p\mid N$, then one can take $M=1296\,c(E)^2$. 
\end{proposition}
\begin{proof}
  We estimate $\lcm(c_p(E_d))_{p \mid Nd}$ where $c_p(E_d)$ is the
  local Tamagawa number at the prime $p$ dividing $Nd$.  If $p \mid
  N$, then $c_p(E_d)$ is either $c_p(E)$ if $w_p=+1$, or $c_p(E^*)$ if
  $w_p=-1$ where $E^*$ is any fixed twist of $E$ by a discriminant
  that is not a square in $\Q_p$. If $p \mid d$, then $c_p(E_d)$ is
  either $1$, $2$ or $4$. Hence, we have
  $$
  \lcm(c_p(E_d))_{p \mid N}\leq 4 \prod_{p \mid N,\, w_p=+1} c_p(E)
  \prod_{p \mid N,\, w_p=-1} c_p(E^*)
  $$
  Now, the result follows using Corollaire 2.2 of \cite{krir} and the
  fact that $|\Delta_{\mbox{{\scriptsize min}}}(E_d)|=
  |d|^6|\Delta_{\mbox{{\scriptsize min}}}(E)|$. 
\end{proof}

\pagebreak[2]

\begin{remark}\ 
  \begin{enumerate}
  \item  With the same techniques, we can obtain similar results for
    $j(E)=0$ or~$1728$.
  \item One can prove (see for example \cite[exercise
    8.17]{silverman}) the following lower bound:
    \begin{equation}\label{useless_bound}
      \hat{h}(G_d) \geq \frac{1}{3} \log |d| + C
    \end{equation}
    where $C$ is some constant depending on $E$. The factor $1/3$ in
    this formula is much better than the factor $1/M$ in
    Proposition \ref{explicite}. However, the constant $C$ (which
    comes from the difference between the naive and the canonical
    heights) is negative and thus the estimate~\eqref{useless_bound}
    is useless for small $d$ (and in practice for all the $d$'s we are
    dealing with).  On the other hand, the estimate of
    Proposition~\ref{explicite} is good enough for our applications
    and has no consequence on the main
    complexity of our method.
  \item The lower bound in Proposition \ref{explicite} is optimal in
    the following sense: suppose that $E$ is given by the equation
    $y^2=P(x)$ where $P(x)$ is a degree $3$ polynomial. Then, one can
    easily check that the point $(rP(r),P(r)^2)$ belongs to
    $E_{P(r)}(\Q)$ and that the height of this point is $\approx 4/3
    \log|P(r)|$.
  \end{enumerate}
\end{remark}

One expect much better lower bounds on average: it is proved in
\cite{delaunay2} that predictions coming from Random Matrix Theory for
derivatives of $L$-functions (see \cite{snaith}) and Cohen-Lenstra
type heuristics for Tate-Shafarevich groups (see \cite{delaunay})
imply that for $k>0$:
\begin{equation}\label{lower}
  \frac{1}{|{\mathcal F}(T)|} \sum_{{\begin{tabular}{c}\\[-6mm]
        $\scriptscriptstyle d\in{\mathcal F}(T) $ \\
        [-2mm]$\scriptscriptstyle L^\prime(E_d,1)\neq 0$
      \end{tabular}}}\hspace*{-4mm} \hat{h}(G_d)^k \gg T^{k/2-\epsilon} 
\end{equation}
where the implied constant depends on $E$, $k$, $\epsilon$ and $w$.

\subsection{Heuristics for the moments of $R(E_d)$}

For $k>0$ we let:
$$
M_k(T)=\frac{1}{|{\mathcal F}(T)|} \sum_{{\begin{tabular}{c}\\[-6mm]
      $\scriptscriptstyle d\in{\mathcal F}(T) $ \\
      [-2mm]$\scriptscriptstyle L^\prime(E_d,1)\neq 0$
    \end{tabular}}}\hspace*{-4mm} R(E_d)^k
$$
Equations \eqref{upper} and \eqref{lower} imply that on average
$\hat{h}(G_d)$ should be of the size of $|d|^{1/2}$. In fact, one can make
similar computations as in \cite{delaunay2} to estimate:
$$
\sum_{{\begin{tabular}{c}\\[-6mm] $\scriptscriptstyle d\in{\mathcal F}(T)
      $ \\ [-2mm]$\scriptscriptstyle L^\prime(E_d,1)\neq 0$
    \end{tabular}}}\hspace*{-4mm} R(E_d)^k S(E_d)^k
$$
Then, Cohen-Lenstra type heuristics for Tate-Shafarevich groups (see
\cite{delaunay}) predict that $\dfrac{1}{|{\mathcal F}(T)|} S(E_d)^k$
tends to a finite limit as $T \rightarrow \infty$ whenever $0 < k
<1$. Therefore, using an empirical argument, we replace the term
$S(E_d)^k$ by a constant and deduce the following heuristics:

\begin{heuristic1}
  For $0<k<1$ we have as $T \rightarrow \infty$:
  \begin{equation}\label{heuristic}
    M_k(T) \sim A_k\; T^{k/2} \log(T)^{k(k+1)/2+a_k-1}
  \end{equation}
  for some constants $A_k$ and $a_k$.
\end{heuristic1}

\noindent The number $a_k$ comes from the contribution of the Tamagawa
numbers in the Birch and Swinnerton-Dyer conjecture.
More precisely we should have:
\begin{itemize}
\item If $E$ (or an isogenous curve) has full rational 2-torsion then
  $a_k=4^{-k}$.
\item If $E$ has exactly one rational 2-torsion point (and no
  isogenous curve has full 2-torsion) then
  $a_k=\frac{1}{2}(4^{-k}+2^{-k})$.
\end{itemize}
For the other cases, we need to make the rather technical assumption
that our restrictions on the discriminants are not 
incompatible with the use of the Chebotarev density theorem (see
\cite{delaunay2}). Then we should have:
\begin{itemize}
\item If $E$ has no rational 2-torsion point and its discriminant is
  not a square then $a_k=\frac{1}{6} \; 4^{-k} + \frac{1}{2} \; 2^{-k}
  + \frac{1}{3}$.
\item If $E$ has no rational 2-torsion point and its discriminant is a
  square then $a_k=\frac{1}{3} \; 4^{-k} + \frac{2}{3}$.
\end{itemize} 
Indeed, the equivalence \eqref{heuristic} depends only on the
isogenous class of the curve, and this explains why we have to
consider the curve in the class with the maximal rational 2-torsion point.

\smallskip

If we restrict our family to negative prime discriminants, the effect
of the Tamagawa numbers disappears and we have $a_k=1$. More precisely
if we let:
$$
{\mathcal F}'=\left\{d<0, \mbox{fund. disc. with}
  \left(\frac{d}{p}\right)=w_p \mbox{ for all } p\mid N  \mbox{ and } |d| \mbox{ is prime } \right\}  
$$
$$
{\mathcal F}'(T)=\Big\{ d \in {\mathcal F}',\; |d|<T \Big\}
$$
and
$$
M_k'(T)=\frac{1}{|{\mathcal F}'(T)|} \sum_{{\begin{tabular}{c}\\[-6mm]
      $\scriptscriptstyle d\in{\mathcal F}'(T) $ \\
      [-2mm]$\scriptscriptstyle L^\prime(E_d,1)\neq 0$
    \end{tabular}}}\hspace*{-4mm} R(E_d)^k \;\; ,
$$
we expect the following heuristic:

\begin{heuristic2}
  For $0<k<1$, we have as $T
  \rightarrow \infty$: 
  \begin{equation}\label{heuristic_prime}
    M_k'(T) \sim A'_k\; T^{k/2} \log(T)^{k(k+1)/2}
  \end{equation}
\end{heuristic2}

\begin{remark}
  These two heuristics are supported by our numerical data for the
  elliptic curves of conductor $N \leq 17$ as we will see in the last
  section.
\end{remark}

The asymptotics \eqref{heuristic} and \eqref{heuristic_prime} imply
that on average the regulators of $(E_d)_{d \in {\mathcal F}}$ behave
as $\approx |d|^{1/2 + \varepsilon}$ suggesting that
$\theta=\varepsilon$ in the Saturday Night Conjecture (see
\cite{CRSW}).  From this we get a density of $T^{1-\varepsilon}$ for
the subset of $d \in {\mathcal F}(T)$ such that $L'(E_d,1)=0$, which
is really surprising compared to the even-rank case. The numerical
data seems to support this fact. On the other hand, extensive
numerical computations by Watkins \cite{watkins} seem to indicate
otherwise. Indeed we want to emphasize that one has always to be
careful with deducing too strong of statements from numerical
investigations.

\section{Computation of generators}

We need to make a certain number of restrictions in order to be able
to apply the method described in this section. First, we assume that
$E$ is the strong Weil curve in its isogeny class (in fact, we just need 
that the Manin's constant of $E$ is equal to $1$) and that $j(E)\neq
0$, $1728$. These are just technical and not essential
assumptions. Furthermore, we assume $L(E,1) \neq 0$ which implies that
$E(\Q)$ has rank $0$ and that $w=+1$. This is a fundamental assumption
and the method would not work without it. Finally, the family of
discriminants ${\mathcal F}$ is obtained by taking $w_p=+1$ for all $p
\mid N$. Hence, $w(E_d)=-1$ and $d$ is a square modulo $4N$ for all
$d\in {\mathcal F}$.

The latter condition implies that one can apply the Heegner point
construction to get a point $P_d \in E(\Q(\sqrt{d}))$ of infinite
order if $L'(E_d,1)\neq 0$.\footnote{Classically the Heegner point
  method is used to construct directly a rational point on $E_d(\Q)$, see
  \cite[Chapter 8.5]{cohen}. However the direct construction of a
  point in a quadratic extension has been already done in connection
  with the problem of congruent numbers by N.~Elkies, see
  \cite{elkies}. The main difference with the construction used in
  this article is that Elkies just wanted a strategy to compute
  efficiently a rational point of some quadratic twists of the
  elliptic curve $32a2$, whereas we want to compute a {\em generator}
  of {\em all} the $E_d(\Q)$ for $d \in {\mathcal F}(T)$ of some
  large $T$. Hence, we really need to be careful in all the steps of
  the method in order to be as efficient as possible. We also have to
  use the full force of the Gross-Zagier  formula and of the Birch
  and Swinnerton-Dyer conjecture in order to get as much information
  as possible all throughout our computations.} For that one has to
evaluate the modular parametrization at well chosen points $\tau \in
X_0(N)$:
$$
\begin{array}{ccccc}
  \varphi \; :&\; X_0(N) & \stackrel{\phi}{\longrightarrow} & \C/\Lambda
  & \stackrel{\wp}{\longrightarrow} E(\C) \\ 
  &    \tau  & \longmapsto & {\scriptstyle \sum\limits_{n\geq1}
    \frac{a(n)}{n} e^{2i\pi n \tau}} 
\end{array}
$$  
with $X_0(N)=\Gamma_0(N) \backslash \overline{\Hy}$ where
$\Gamma_0(N)$ is the congruence subgroup of $SL_2(\Z)$ of matrices
with lower left entry divisible by $N$, $\overline{\Hy}=\Hy \cup \Q$
is the completed upper half plane, $\Lambda$ is the period lattice
associated to $E$ and $\wp$ is the analytic isomorphism given by the
Weierstrass function (and its derivative).

\subsection{Description of the method}\label{algo}

We now briefly describe the algorithm step by step.

\medskip

\noindent{\bf STEP 1.}\begin{it}
  For each ideal class ${\mathcal C}$ in the
  class group $\Cl(d)$ of $\Q(\sqrt{d})$, we choose an integral ideal
  ${\mathfrak a} \in {\mathcal C}$ such that:
  \begin{equation}\label{heegner_condition}
    {\mathfrak a}=A\; \Z + \frac{-B +\sqrt{d}}{2} \; \Z \;\; \mbox{ with }
    N \mid A \mbox{ and } B\equiv \beta \pmod{2N} 
  \end{equation}
  where $\beta=\beta_d$ is a fixed integer such that $\beta^2 \equiv d
  \pmod{4N}$. \\
  Then, to ${\mathcal C} = [{\mathfrak a}]$, we associate the Heegner
  point:  
  $$
  \tau_{[{\mathfrak a}]} =\frac{-B+\sqrt{d}}{2A}
  $$
\end{it}

\smallskip

\noindent{\sc Comments.} The point $\tau_{[{\mathfrak a}]}$ lies in
the upper half plane and is a well defined point in $X_0(N)$.
Nevertheless, in order to make the computations as easy as possible,
we need to choose ${\mathfrak a}$ such that $A$ is as small as
possible. Using classical algorithms (see \cite{cohen0}), we can
compute a set of ideals $\{{\mathfrak a}_i\}_i$ representing all the
classes of $\Cl(d)$:
$$
{\mathfrak a}_i=a_i\Z+\frac{-b_i+\sqrt{d}}{2}\Z
$$
with $0<a_i\ll|d|^{1/2}$ where the implied constant is explicit. We
can assume without loss of generality that the $a_i$'s are relatively
prime with $N$. Then, the ideals 
${\mathfrak a_i}\overline{{\mathfrak n}}$
satisfy~\eqref{heegner_condition} where  
$$
\overline{\mathfrak n}=N\;\Z+\frac{\beta - \sqrt{d}}{2} \Z
$$
From this it follows that one can choose the ideals 
${\mathfrak a}_i$'s in such a way that we have the following lower
bound: 
\begin{equation}\label{grand}
  \Im(\tau_{[{\mathfrak a}_i]}) \gg 1/N
\end{equation}
The complexity of this step is thus dominated by the class number of
$\Q(\sqrt{d})$, hence is at most $O(|d|^{1/2} \log |d|)$.

\medskip

\noindent{\bf STEP 2.}\begin{it} We compute
  $$
  z_d = \sum_{[{\mathfrak a}]}\phi(\tau_{[{\mathfrak a}]})
  $$ 
  where the sum is over the classes of $\Cl(d)$, and then a complex
  approximation of $P_d=\wp(z_d) \in E(\C)$.  The theory of complex
  multiplication and of Heegner points imply that $P_d\in
  E(\Q(\sqrt{d}))$. Using this approximation, we try to recognize the
  four rational numbers $r_1$, $s_1$, $r_2$ and $s_2$ such that
  $P_d=(r_1+s_1\sqrt{d},r_2+s_2\sqrt{d})$ and test if $P_d$ is a point
  of infinite order.
\end{it}

\smallskip

\noindent {\sc Comments.} This is the main step of the method. Note
that one can reduce the number of evaluations of $\phi$ by $2$ using
the following trick. Once we have already computed
$\varphi(\tau_{[{\mathfrak a}]})$, since $w=+1$ we can deduce from it
$\varphi(\tau_{[{\mathfrak a^{-1}n}]})$ using the formula:
\begin{equation}\label{phi_phibar}
  \overline{\varphi(\tau_{[{\mathfrak
        a}]})}=-\varphi(\tau_{[{\mathfrak a^{-1}n}]}) +Q 
\end{equation}
where $Q$ is an explicit rational torsion point in $E(\Q)$ depending
only on $E$.

Given a complex number $\tilde{x}_{P_d}$ that is an approximation of the
$x$-coordinate $x_{P_d}$ of the point $P_d$ computed as explained
above, we need to recover from it the two rational numbers $r_1$ and
$s_1$ such that $x_{P_d} = r_1 + s_1\sqrt{d}$. Note that for candidate
values $r_1$ and $s_1$, one can check if they are indeed correct by
trying to compute two rationals $r_2$ and $s_2$ such that
$(r_1+s_1\sqrt{d},r_2+s_2\sqrt{d}) \in E(\Q(\sqrt{d}))$. Let
$\tilde{r} = \Re(\tilde{x}_{P_d})$ and $\tilde{s} =
\Im(\tilde{x}_{P_d})/\sqrt{d}$. For $e \geq 1$ we look for a small
integral relation (using the LLL-algorithm) between the columns $C_1,
C_2, C_3$ of the matrix
$$
\left(
\begin{array}{ccc} 
  -10^e &   0    & \lfloor 10^e \, \tilde{r} \, \rceil \\
    0   & -10^e  & \lfloor 10^e \, \tilde{s} \, \rceil \\
    0   &   0    &   1
\end{array}
\right)
$$
where $\lfloor . \rceil$ denotes the closest integer. Indeed, for such
a relation, say
$$
\lambda_1 C_1 + \lambda_2 C_2 + \lambda_3 C_3
$$
of norm $M$, we have that $\lambda_1/\lambda_3$, resp.
$\lambda_2/\lambda_3$, is an approximation of $\tilde{r}$, resp.
$\tilde{s}$, with an error less than $\sqrt{M}/10^e$, and the
denominator $\lambda_3$ is smaller (in absolute value) than
$\sqrt{M}$. In order for this method to work, we need to compute
$\tilde{r}$ and $\tilde{s}$ at a suitably large enough precision and
to choose $e$ accordingly. More precisely, to recognize $x_{P_d}$ as
an element of $\Q(\sqrt{d})$ we need about $\hat{h}(P_d)$ digits.
Bounding the coefficients $a(n)/n$ by $1$ in the sum defining $\phi$
and using (\ref{grand}), we see that we need to sum approximatively
$\hat{h}(P_d)$ coefficients for $\phi$.  The Gross-Zagier theorem
\cite{gross-zagier} asserts that:\footnote{Actually, the Gross-Zagier
  theorem only applies for odd $d$'s. For even $d$'s the formula is a
  conjecture of Hayashi \cite{hayashi}.}
\begin{equation}\label{gross-zagier}
  \hat{h}(P_d)=\frac{L(E,1)\,L'(E_d,1)\,\sqrt{|d|}}{4\,\vol(E)}
\end{equation}
Applying the Lindel\"of hypothesis we deduce that $\hat{h}(P_d) \ll
|d|^{1/2+\varepsilon}$. Hence, the complexity of this step is $\ll
|d|^{1/2+\varepsilon} |\Cl(d)| \ll |d|^{1+\varepsilon}$. This step can
fail in two ways. First case: the computation has not been done to a
large enough precision. In that case we have to increase the precision
and start over. Second case: the point $P_d$ is a torsion point and in
that case $L'(E_d,1)=0$. If we suspect $P_d$ to be in fact a torsion
point, we can compute directly an approximation of $L'(E_d,1)$ and
prove that it is indeed zero using the following proposition (whose
proof we postpone to after the proof of the next proposition).

\begin{proposition}\label{lower_bound} 
  If 
  $$
  L'(E_d,1) \leq \frac{\vol(E)}{1296\, c(E)^2 \, L(E,1)} \; |d|^{-1/2}
  \log |d|
  $$ 
  then $L'(E_d,1)=0$.
\end{proposition}

\medskip

\noindent{\bf STEP 3.} \begin{it} If $P_d$ is a point of infinite
  order, i.e. {\bf STEP 2} has succeeded, then the point
  $R_d=\psi(P_d-\overline{P_d})$ is a point of infinite order in
  $E_d(\Q)$. We divide it in the Mordell-Weil group $E(\Q)$ until we
  get a generator $G_d \in E_d(\Q)$. We define the integer $\ell_d$ by
  $R_d = \ell_d \, G_d \pmod{E_d(\Q)_{\mbox{{\scriptsize tor}}}}$.
\end{it}

\smallskip

\noindent {\sc Comments.} The point $R_d$ is rational since
$R_d=\psi(P_d) + \overline{\psi(P_d)}$ by \eqref{conj_de_psi}. If
$L'(E_d,1) \neq 0 $ then we know that $G_d$ is a generator of
$E_d(\Q)$ modulo torsion.

\begin{proposition}
  $\hat{h}(R_d)=4\,\hat{h}_E(P_d)$, hence $R_d$ is non-torsion if and
  only if $P_d$ is non-torsion \textup{(}that is if and only if
  $L'(E_d,1)\neq 0$\textup{)}.
\end{proposition}

\begin{proof}
  The height does not depend on the model of the elliptic curve, hence
  $\hat{h}(R_d)=\hat{h}_E(P_d-\overline{P_d})$. Furthermore, equation
  \eqref{phi_phibar} implies that $P_d=-\overline{P_d}$ plus a
  rational torsion point.
\end{proof}
\begin{proof}[Proof of proposition \ref{lower_bound}.] We use the
  lower bound from proposition \ref{explicite} for $\hat{h}(R_d)$ and
  equation (\ref{gross-zagier}).
\end{proof}

From Proposition \ref{explicite} we know that:
\begin{equation}\label{borne_l}
  |\ell_d| < 36 \, c(E) \sqrt{\frac{\hat{h}(R_d)}{\log|d|}} \ll
  |d|^{1/4+\varepsilon} 
\end{equation}

Hence there are finitely many primes $p$ for which we need to check
$p$-di\-vi\-si\-bi\-li\-ty. Also, it is well-known that
$E_d(\Q)_{\mbox{{\scriptsize tors}}}$ does not depend upon $d$ (for
all $d$'s except at most one) and can only be $\simeq \{0\}$, $\Z/2\Z$
or $\Z/2\Z \times \Z/2\Z$. Therefore we need to be careful about
torsion only when we consider $2$-divisibility which can be tested
easily using $2$-division polynomial. For an odd prime $p$, we use the
following method to rule out $p$-divisibility. We find a prime $r$, of
good reduction, such that the order $\alpha$ of the group $E_d(\F_r)$
is divisible by $p$. Then if $(\alpha/p) R_d$ is not zero in
$E_d(\F_r)$, we know that $R_d$ is not divisible by $p$ in $E(\F_r)$,
and thus in $E_d(\Q)$ too. If after having performed a large number of
such tests, we have not been able to prove that $R_d$ is not divisible
by $p$, then we ``know'' that the point must be divisible by $p$ and
we perform the division.\footnote{Indeed, in all cases, either we
  could prove that the point is not divisible by $p$ by such a test,
  or we could actually divide it by $p$.}

\medskip

\noindent{\bf STEP 4.}\begin{it} We compute the regulator of $E_d$ {\rm
    (}in the rank 1 case{\rm )} which is equal to $\hat{h}(G_d)$ and
  the order of the Tate-Shafarevich group $|\TSg(E_d)|$.
\end{it}

\smallskip

\noindent {\sc Comments.} We can compute the order of $|\TSg(E_d)|$
using:
\begin{proposition}\label{formule_sha} Under the Birch and
  Swinnerton-Dyer conjecture\footnote{For even $d$'s, we need again to
    assume the conjecture of Hayashi \cite{hayashi}.}, the following
  equality holds   
  $$
  |\TSg(E_d)|=\frac{|E(\Q)_{\mbox{{\scriptsize tor}}}|^2 \,
    |E_d(\Q)_{\mbox{{\scriptsize tor}}}|^2}{|\TSg(E)|\,c(E)^2} \;
  \frac{\ell_d^2}{2sg(\Delta_{\mbox{{\scriptsize {\em min}}}}(E))\,
    \delta_8(d,c_4) \prod\limits_{p\mid d}c_p(E_d)}
  $$
  where $sg(x)=1$ if $x<0$ and $sg(x)=2$ otherwise. 
\end{proposition}
\begin{proof} Indeed, we have:
$$
\ell_d^2 \, \hat{h}(G_d) = 4\hat{h}(P_d) 
= \frac{L(E,1) \, L'(E_d,1) \, \sqrt{|d|}}{\vol(E)}
$$
Now we replace $L(E,1)$ and $L'(E_d,1)$ by the values predicted by the
Birch and Swinnerton-Dyer conjecture. After simplifying the regulator
$\hat{h}(G_d)$ on both sides, we get:
$$
\ell_d^2 = \frac{|\TSg(E_d)| \, |\TSg(E)| \, c(E)}
                {|E(\Q)_{\mbox{{\scriptsize tor}}}|^2 \, 
                 |E_d(\Q)_{\mbox{{\scriptsize tor}}}|^2} 
           \cdot c(E_d) \cdot
           \frac{\Omega(E) \, \Omega(E_d) \, \sqrt{|d|}}{\vol(E)} 
           (\times 4 \mbox{ if } \Delta_{\mbox{{\scriptsize {\em min}}}}(E) >0)
$$
Since for all $p\mid N$ we have $w_p=+1$, the curves $E$ and $E_d$ are
isomorphic over $\Q_p$ and thus $c_p(E_d)=c_p(E)$. So $c(E_d)=c(E) \,
\prod_{p \mid d} c_p(E_d)$. Finally a computation of the periods of
$E_d$ shows that:
$$
\frac{\Omega(E) \, \Omega(E_d) \, \sqrt{|d|}}{\vol(E)}
= \left\{
  \begin{array}{cc} 2\,\delta_8(d,c_4) & \mbox{ if }
    \Delta_{\mbox{{\scriptsize min}}}(E)<0 \\ \delta_8(d,c_4) & \mbox{
      if } \Delta_{\mbox{{\scriptsize min}}}(E)>0
  \end{array} \right.
$$
and the proposition follows. 
\end{proof}

\begin{remark} 
  The order of the Tate-Shafarevich group is a square, therefore the
  proposition implies that the following quantity must be a square: 
  $$
  2 \, sg(\Delta_{\mbox{{\scriptsize {\em min}}}}(E))\,
  \delta_8(d,c_4) \prod\limits_{p\mid d}c_p(E_d)
  $$ 
\end{remark}

From the above we see that for each individual $d$ the complexity for
computing $\hat{h}(G_d)$ and $|\TSg(E_d)|$ is at worst
$O(|d|^{1+\varepsilon})$. From these values we can deduce the value of
$L'(E_d,1)$ at arbitrary precision. Note that the direct computation
of $L'(E_d,1)$ by the rapidly converging series needs also $O(|d|)$
terms.\footnote{More generally, in order to compute $L'(E,1)$ for an
  elliptic curve $E$, one needs to sum the first $O(\sqrt{N})$ terms of
  the series, where $N$ is the conductor of $E$, and the constant in
  the ``$O$'' depends on the required accuracy.} Nevertheless, for
large precisions, in practice, it is often much more efficient to
compute $L'(E_d,1)$ as a by product of our computations than to
evaluate it directly.  This is probably due to the fact (see the
discussion on the computations) that the implied constant is small in
the prediction $M_1(T) =O(T^{1/2}(\log T)^a)$.
\subsection{An example}
We take $E=11a1 \; : \; y^2 + y = x^3-x^2-10x-20$ and $d=-79$ so that
the curve $E_d$ has minimal equation:
$$
E_d \; : \; y^2 + y = x^3 + x^2 -64490 x + 11396008
$$
We take $\beta=3$ so that $\beta^2 \equiv -79 \pmod{44}$. The class
group $\Cl(-79)$ of $\Q(\sqrt{-79})$ is cyclic of order 5, and the ideals:
$$\begin{array}{ll}
{\mathfrak a} = 11\Z + \dfrac{-3 + \sqrt{-79}}{2} \Z  \; , &
{\mathfrak b} = 22\Z + \dfrac{-3 + \sqrt{-79}}{2} \Z  \; , \\[5pt]
{\mathfrak c} = 44\Z + \dfrac{-3 + \sqrt{-79}}{2} \Z  \; ,&
{\mathfrak a}^{-1}{\mathfrak n} \quad\text{and}\quad {\mathfrak
  b}^{-1}{\mathfrak n} 
\end{array}$$
where 
$$
\mathfrak{n} = N \Z + \frac{\beta + \sqrt{d}}{2}\Z = 11\Z +
\frac{3 + \sqrt{-79}}{2}\Z 
$$
form a complete set of representatives of the ideal class group. We
compute  
$$
z = 2\Re\big(\phi(\tau_{[{\mathfrak a}]})+\phi(\tau_{[{\mathfrak
    b}]})\big) + \phi(\tau_{[{\mathfrak c}]}) \in \C/\Lambda 
$$
and we find
$$
P_d=\wp(z) \approx (-3.59000\cdots +
0.22000\cdots\sqrt{-79},5.17600\cdots +0.61600\cdots \sqrt{-79})
$$ 
so we easily recognize 
$$
P_d=\left(\frac{-179 +11\sqrt{-79}}{50},
  \frac{647+77\sqrt{-79}}{125}\right) \in E(\Q(\sqrt{-79}))
$$
From this, we get the point $R_d=\psi(P_d)+\overline{\psi(P_d)}
=(47,2910) \in E_d(\Q)$. And Formula~\eqref{borne_l} says that
$|\ell_d| \leq 293$ where $G_d = \ell_d R_d$. We find that the point
$R_d$ is divisible by $2$, more precisely $R_d=-2(26,3120)$, so that
$(26,3120) = \ell'_d G_d$ with $|\ell'_d| \leq 73$. We then easily
check that the point $(26,3120)$ is not divisible by any prime $\leq
73$ in the group $E_d(\Q)$, hence one can take $G_d=(26,3120)$ and
$|\ell_d|=2$. Proposition~\eqref{formule_sha} gives:
$$
|\TSg(E_d)|= 1
$$

\section{Discussion and numerical data}

We have computed, using the method described in the previous section,
the regulators and the order of the Tate-Shafarevich groups of the
twists $E_d$ of $E$ of the four elliptic curves $11a1$, $14a1$, $15a1$
and $17a1$, and for all available discriminants $d \in {\mathcal
  F}(1.5\times 10^6)$ with $w_p=+1$ for all $p \mid N$. We discuss in
this section the data we obtained and compare it with the
heuristics. All the computations have been performed using the PARI/GP
system \cite{pari} and the data is available at
\begin{center}
\texttt{http://math.univ-lyon1.fr/\~{ }roblot/tables.html} 
\end{center}

\medskip

We begin with the curves of prime conductor ($11a1$ and $17a1$) since
for the last two curves ($14a1$ and $15a1$), the congruence conditions
are more restrictive and therefore the number of discriminants in
${\mathcal F}(1.5\times 10^6)$ is quite small compared to $1.5\times
10^6$.

% Precisely, for each of these four curves, we give the number of
% discriminants we considered (i.e. $|{\mathcal F}(1.5\times 10^6)|$),
% the form of the heuristic for the moment of the regulators, the
% number of $d \in {\mathcal F}(1.5\times 10^6))$ such that
% $L'(E_d,1)=0$. Then we ploth the graph comparing the moment of the
% regulators and the heuristic etc...

\subsection{The curve 11a1}
The curve $E$ is defined by $ E \; : \; y^2 + y = x^3-x^2-10x-20$. It
has conductor $N = 11$ and rank $0$ over $\Q$. We have $w_{11}=+1$.

\subsubsection{Numerical results for all discriminants}

\begin{figure}[htbp]
  \begin{small}
    \centering
    \hspace*{-1.5cm}\mbox{
      \subfigure{
        \includegraphics[width=8cm]{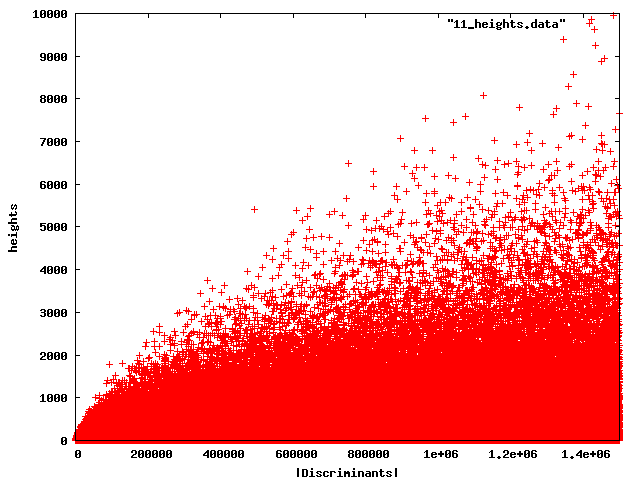}
      }
      \subfigure{
        \includegraphics[width=8cm]{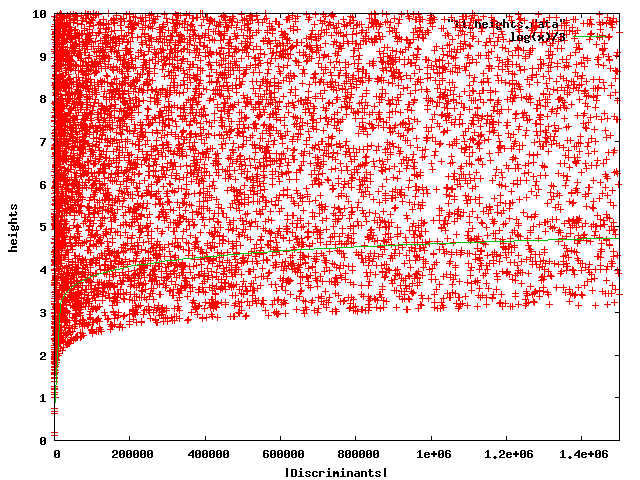}
      }
    }
    \vskip-10pt\caption{\label{11_heights} Regulators of the twists of 
      $11a1$. On the left, all the regulators are plotted, and on the
      right, only regulators less than $10$.  The gap between the
      $x$-axis and the minimal heights is clearly visible on the
      right.}
    \vspace*{.5cm}
    \hspace*{-1.5cm}\mbox{
      \subfigure{
        \includegraphics[width=8cm]{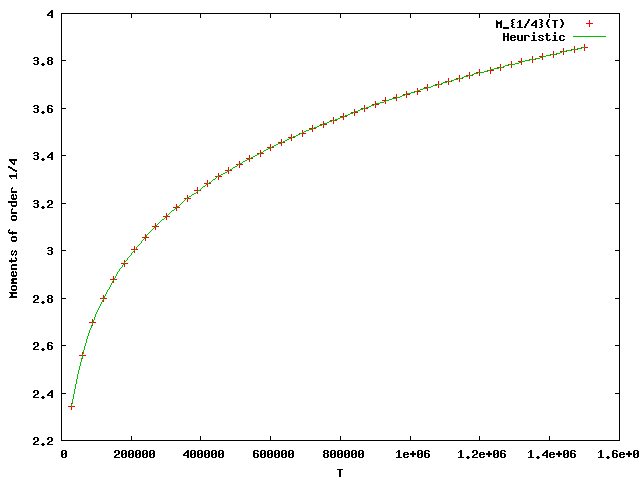}
      }
      \subfigure{
        \includegraphics[width=8cm]{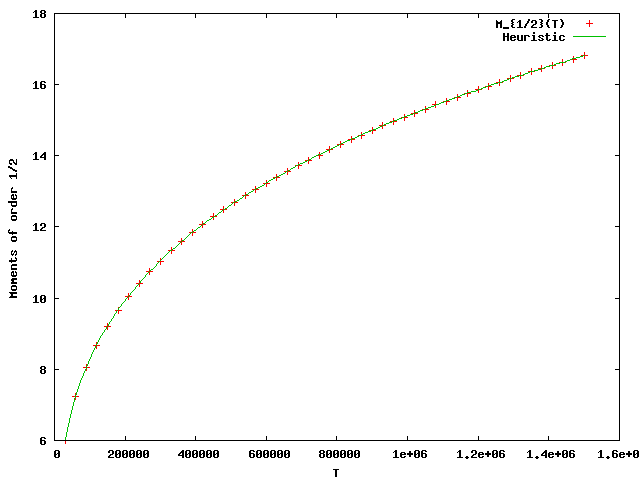}
      }
    }
    \hspace*{-1.5cm}\mbox{
      \subfigure{
        \includegraphics[width=8cm]{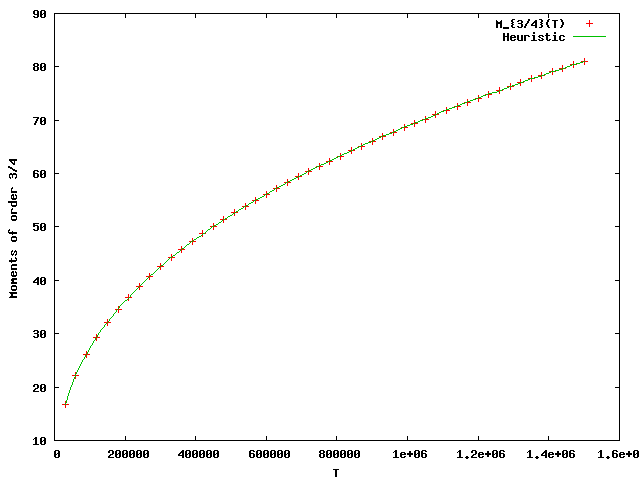}
      }
      \subfigure{
        \includegraphics[width=8cm]{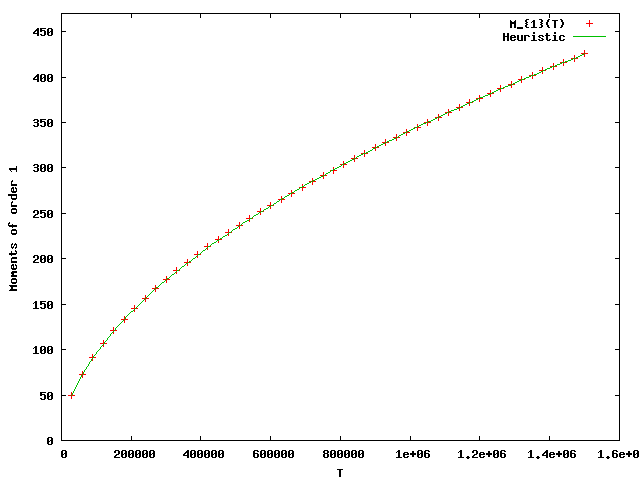}
      }
    }
    \vskip-10pt\caption{\label{11_moment} Moments of order
      $\frac{1}{4}$, $\frac{1}{2}$, $\frac{3}{4}$ and $1$ of the
      regulators of the twists of $11a1$ and the function given by the
      heuristics.} 
  \end{small}
\end{figure}

\begin{figure}[htpb]
  \begin{small}
    \hspace*{-1.5cm}\begin{minipage}[t]{8cm}
      \includegraphics[width=8cm]{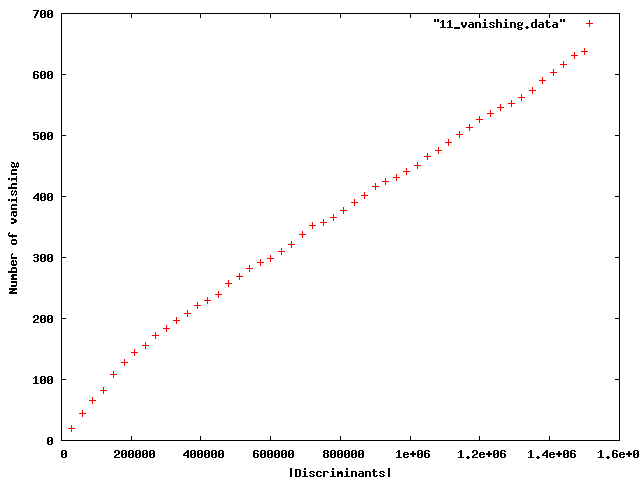}
      \vskip-10pt\caption{\label{11_extra_vanishing}Number of
        extra-vanishing of $L'(E_d,1)$ for $E=11a1$.} 
    \end{minipage}
    \hskip.3cm
    \begin{minipage}[t]{8cm}
      \includegraphics[width=8cm]{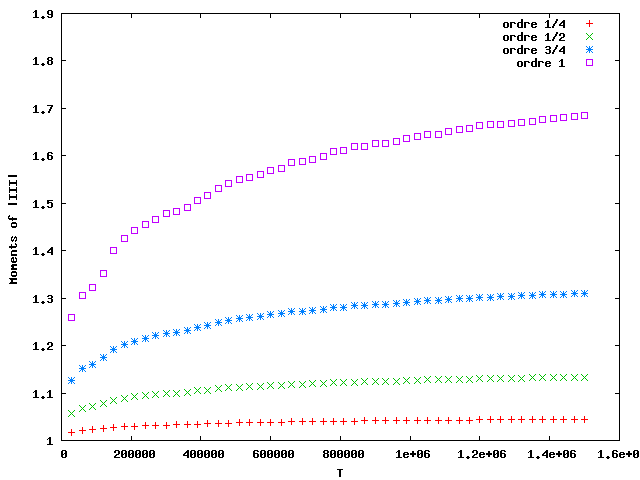}
      \vskip-10pt\caption{\label{11_sha} Moments of order
        $\frac{1}{4}$, $\frac{1}{2}$, $\frac{3}{4}$ and $1$ for the
        order of the Tate-Shafarevich groups of the twists of
        $11a1$. The heuristics suggest that the moments of order $k<1$
        tend to a constant (depending on $k$) whereas the moment of
        order $1$ should tend to infinity.} 
    \end{minipage}  
    \vskip1cm
    \hspace*{-1.5cm}\mbox{
      \subfigure{
        \includegraphics[width=8cm]{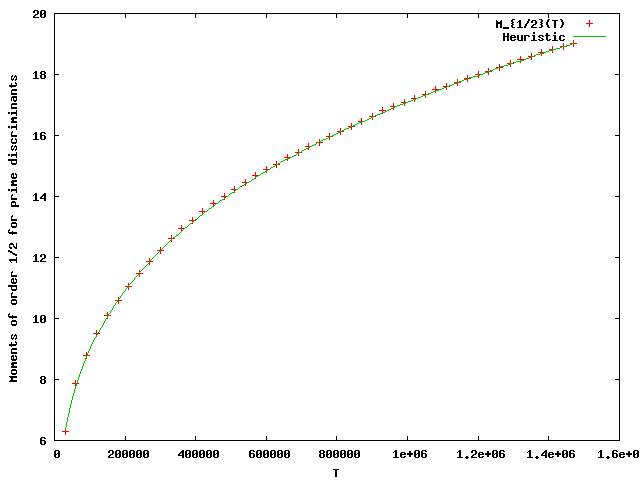}
      }
      \subfigure{
        \includegraphics[width=8cm]{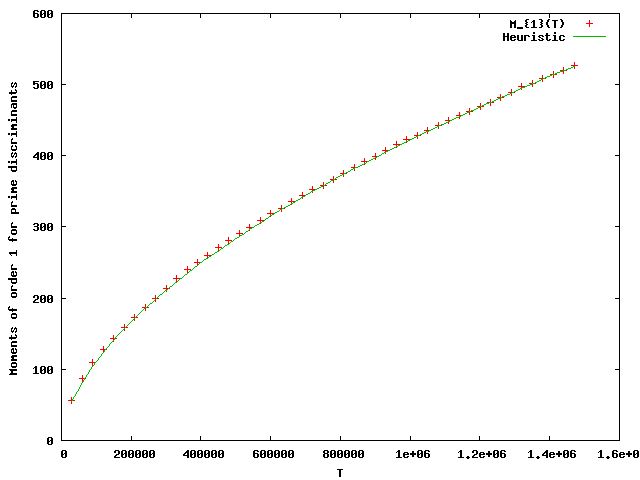}
      }
    }
    \vskip-10pt\caption{\label{11_premier} Moments of order
      $\frac{1}{2}$ and $1$ for the regulators of the twists of $11a1$
      by prime discriminants and the functions given by the
      heuristics.}  
  \end{small}
\end{figure}

\begin{itemize}
\item Number of discriminants: $|{\mathcal F}(1.5\times 10^6)| =
  208977$.
\item Largest regulator: $\approx 9945$ (for $d=-1482139$).
\item Number of extra-vanishing: 638.
\end{itemize}

We have $E(\Q)_{{\rm tors}} \simeq \Z/5\Z$ and there is no curve in
its isogeny class having rational 2-torsion.  Hence the heuristics
predict that:
$$
M_k(T) \sim A_k T^{k/2} \log(T)^{\frac{k(k+1)}{2}+\frac{1}{6 \cdot
    4^k} +\frac{1}{2 \cdot 2^k} -\frac{2}{3}}
$$
for some constant $A_k$. We computed $A_k$ numerically to fit the data
(values found: $A_{1/4} \approx 0.60$, $A_{1/2} \approx 0.33$,
$A_{3/4} \approx 0.16$, $A_{1} \approx 0.07$) and we plot the graph of
the function given by the heuristics and the points $(T,M_k(T))$ for
$T=1,2,\dots,150 \times 3\cdot 10^4$ and for $k=1/4$, $1/2$, $3/4$ and
$1$. As it can been seen the graphs (see Figure~\ref{11_moment}) are
in close agreement.

\subsubsection{Numerical results for prime discriminants}

\begin{itemize}
\item Number of prime discriminants:  28535.
\item Largest regulator: $\approx 9250$ (for $d=-1433539$).
\item Number of extra-vanishing: 0.\footnote{There is no
    extra-vanishing in this case using the results of
    \cite{antoniadis-frey}.}
\end{itemize}

The heuristics for prime discriminants predict that:
$$
M'_k(T) \sim A'_k T^{k/2} \log(T)^{k(k+1)/2}
$$
for some constant $A'_k$. We computed $A'_k$ numerically to fit the
data (values found: $A'_{1/2} \approx 0.20$, $A'_{1} \approx 0.03$)
and we plot the graph of the function given by the heuristics and the
points $(T,M'_k(T))$ for $T=1,2,\dots 150 \times 3 \cdot 10^4$, and
$k=1/2$, $1$ (see Figure~\ref{11_premier}).

\subsection{The curve 17a1}

The curve $E$ is defined by $E \; : \; y^2+xy+y=x^3-x^2-x-14$. It has
conductor 17 and rank $0$ over $\Q$. We have $w_{17}=+1$.

\subsubsection{Numerical results for all discriminants}

\begin{itemize}
\item Number of discriminants: 215305.
\item Largest regulator: $\approx 31746$ (for $d=-1257787$).
\item Number of extra-vanishing: 1140.
\end{itemize}

\begin{figure}[htbp]
  \begin{small}
    \hspace*{-1.5cm}\mbox{
      \subfigure{
        \includegraphics[width=6cm]{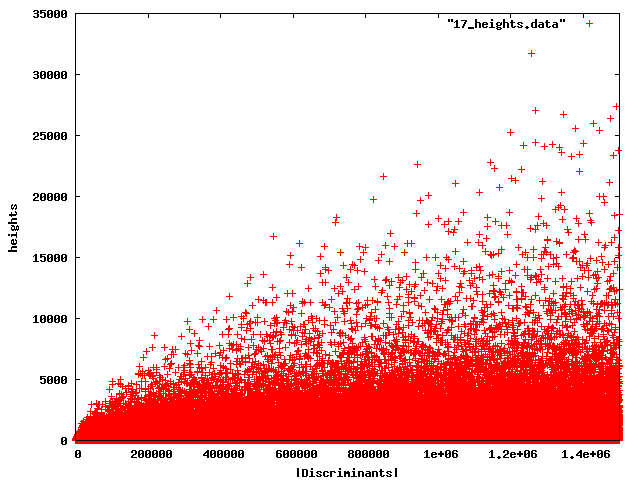}
      }
      \subfigure{
        \includegraphics[width=6cm]{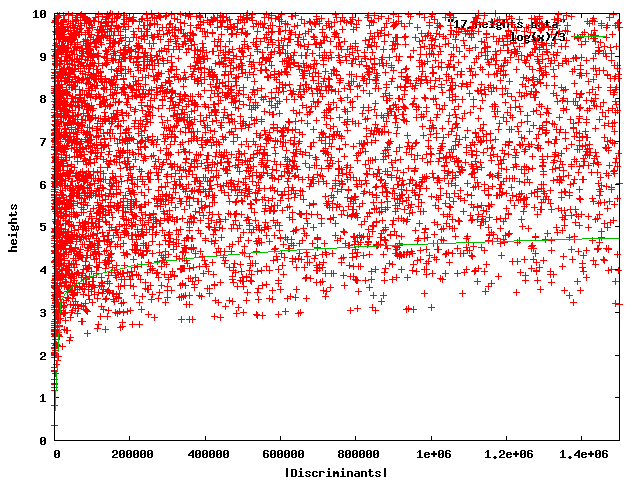}
      }
    }
    \vskip-10pt\caption{\label{17_heights} Regulators of the twists of
      $17a1$. On the left, all regulators are plotted, and on the
      right only regulators less than 10. The gap between the $x$-axis
      and the minimal heights is cleary visible on the right.}
    \vspace*{.5cm}
    \hspace*{-1.5cm}\mbox{
      \subfigure{
        \includegraphics[width=6cm]{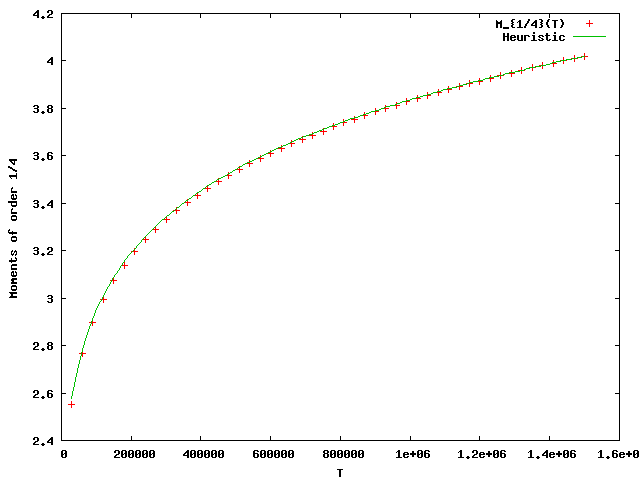}
      }
      \subfigure{
        \includegraphics[width=6cm]{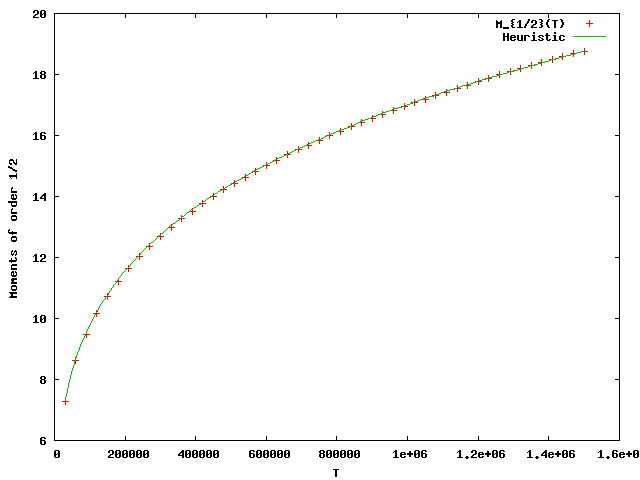}
      }
    }
    \hspace*{-1.5cm}\mbox{
      \subfigure{
        \includegraphics[width=6cm]{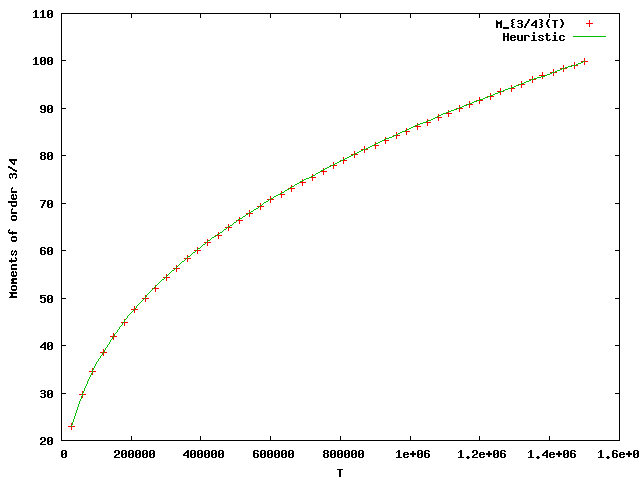}
      }
      \subfigure{
        \includegraphics[width=6cm]{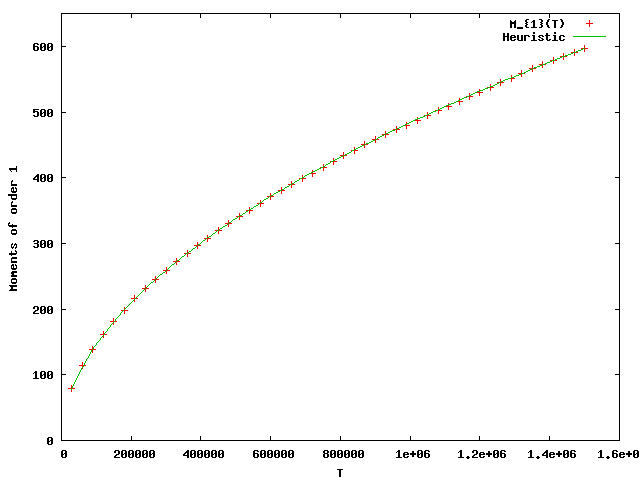}
      }
    }
    \vskip-10pt\caption{\label{17_moment} Moments of order $\frac{1}{4}$,
      $\frac{1}{2}$, $\frac{3}{4}$ and $1$ of the regulators of the
      twists of $17a1$ and the functions given by the heuristics.}
  \end{small}
\end{figure}

\begin{figure}[htpb]
  \begin{small}
    \hspace*{-1.5cm}\begin{minipage}[t]{8cm}
      \includegraphics[width=6cm]{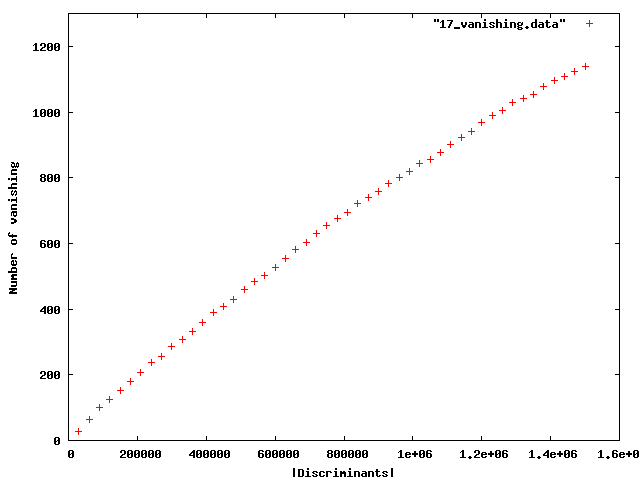}
      \vskip-10pt\caption{\label{17_extra_vanishing}Number of
        extra-vanishing of $L'(E_d,1)$ for $E=17a1$.} 
    \end{minipage}
    \hskip.3cm
    \begin{minipage}[t]{8cm}
      \includegraphics[width=6cm]{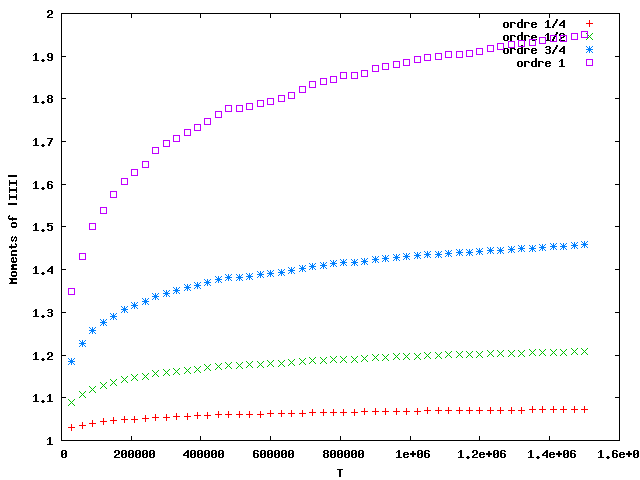}
      \vskip-10pt\caption{\label{17_sha} Moments of order
        $\frac{1}{4}$, $\frac{1}{2}$, $\frac{3}{4}$ and $1$ for the
        order of the Tate-Shafarevich groups of the twists of
        $17a1$. The heuristics suggest that the moments of order $k<1$
        tend to a constant (depending on $k$) whereas the moment of
        order $1$ should tend to infinity.} 
    \end{minipage}  
    \vskip1cm
    \hspace*{-1.5cm}\mbox{
      \subfigure{
        \includegraphics[width=6cm]{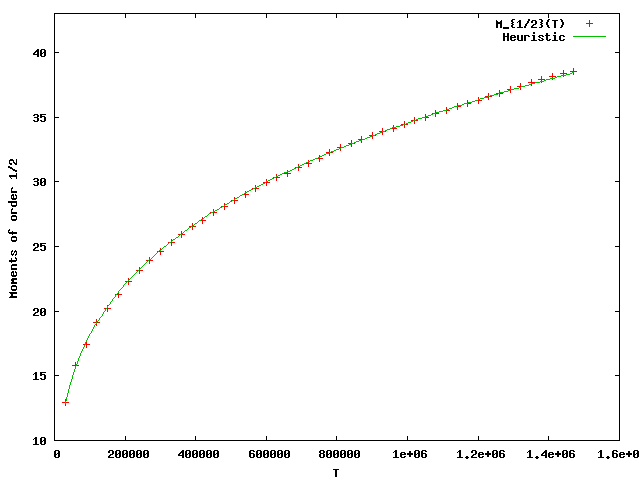}
      }
      \subfigure{
        \includegraphics[width=6cm]{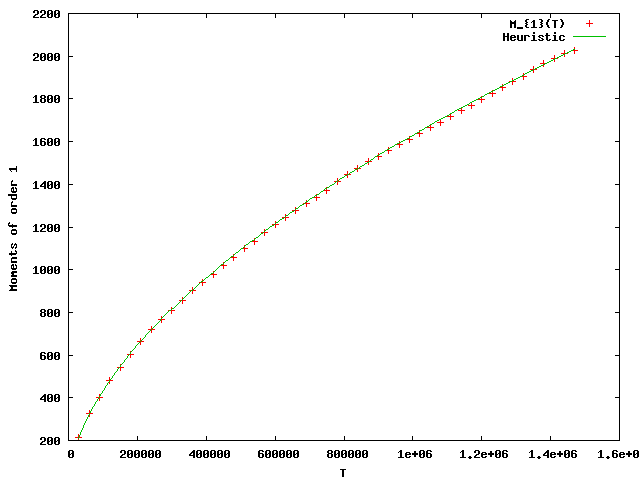}
      }
    }
    \vskip-10pt\caption{\label{17_premier} Moments of order
      $\frac{1}{2}$ and $1$ for the regulators of the twists of $17a1$
      by prime discriminants and the functions given by the
      heuristics.} 
  \end{small}
\end{figure}

\begin{remark}
  Note that the graphs of extra-vanishing for the curves 11a1
  (Figure~\ref{11_extra_vanishing}) and 17a1
  (Figure~\ref{17_extra_vanishing}) suggest that the density of
  extra-vanishing is larger for the twists of 17a1 than for those of
  11a1. However the asymptotic for the moments of the regulators is
  smaller (as $T\rightarrow \infty$) for 17a1 than for 11a1 which suggest that
  there are more constraints on the regulators of the twists of 11a1
  and thus imply in turn that we should have more extra-vanishing for
  this family. In fact, the constants $A_k$ in the asymptotics of
  $M_k(T)$ are larger for the curve 17a1, but asymptotics of the
  functions $M_k(T)$ for the curve  11a1 are larger than for the
  curve 17a1 for very large values of $T$ that are completely out of
  reach for computations. Therefore our guess is that the density of
  extra-vanishing for the twists of 11a1 will become greater than that
  for the twists of 17a1 for those very large values.
\end{remark}

The curve $17a2$ has full rational 2-torsion, hence the heuristics
predict that
$$
M_k(T) \sim A_k T^{k/2} \log(T)^{\frac{k(k+1)}{2}+\frac{1}{4^k}-1}
$$
for some constant $A_k$. We computed $A_k$ numerically to fit the data
(values found: $A_{1/4} \approx 0.97$, $A_{1/2} \approx 0.75$,
$A_{3/4} \approx 0.47$, $A_{1} \approx 0.25$ and we plot the graph of
the function given by the heuristics and the points $(T,M_k(T))$ for
$T=1,2,\dots 150 \times 3 \cdot 10^4$, and $k=1/4$, $1/2$, $3/4$ and
$1$ (see Figure~\ref{17_moment}).

\subsubsection{Numerical results for prime discriminants}

\begin{itemize}
\item Number of prime discriminants: 28601.
\item Largest regulator: $\approx 31745$ (for $d=-1257787$).
\item Number of extra-vanishing: 0.\footnote{There is no
    extra-vanishing in this case using the results of
    \cite{antoniadis-frey}.}
\end{itemize}

The heuristics for prime discriminants predicts that: 
$$
M'_k(T) \sim A'_k T^{k/2} \log(T)^{k(k+1)/2}
$$
for some constant $A'_k$. We computed $A'_k$ numerically to fit the
data (values found: $A'_{1/2} \approx 0.41$, $A'_{1} \approx 0.12$)
and we plot the graph of the function given by the heuristic and the
points $(T,M'_k(T))$ for $T=1,2,\dots 150 \times 3 \cdot 10^4$,
and $k=1/2$, $1$ (see Figure~\ref{17_premier}).

\subsection{The curve 14a1}

The curve $E$ is defined by $E \; : \; y^2+xy+y=x^3+4x-6$. It has
conductor $N = 14$ and rank $0$ over $\Q$. We have $w_{2} = w_{7} =
+1$.

\begin{itemize}
\item Number of discriminants:  66516.
\item Largest regulator: $\approx 16937$ (for $d=-1416631$).
\item Number of extra-vanishing: 262. 
\end{itemize}

\begin{figure}[hbtp]
  \begin{small}
    \centering
    \hspace*{-1.5cm}\mbox{
      \begin{minipage}[t]{8cm}
        \includegraphics[width=6cm]{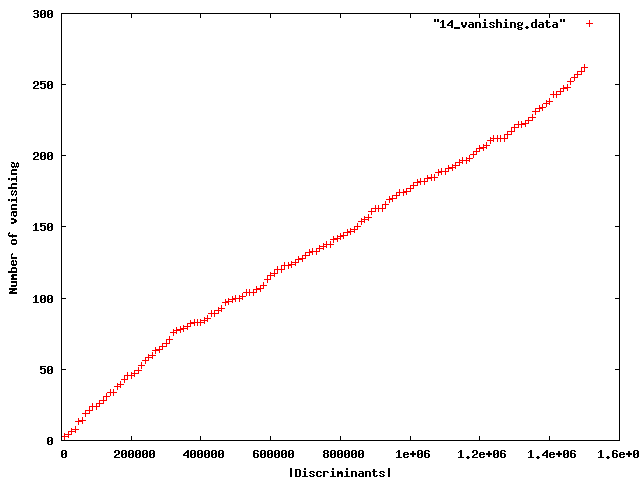}
        \vskip-10pt\caption{\label{14_extra_vanishing}Number of
          extra-vanishing of $L'(E_d,1)$ for $E=14a1$.} 
      \end{minipage}
      \hskip.3cm
      \begin{minipage}[t]{8cm}
        \includegraphics[width=6cm]{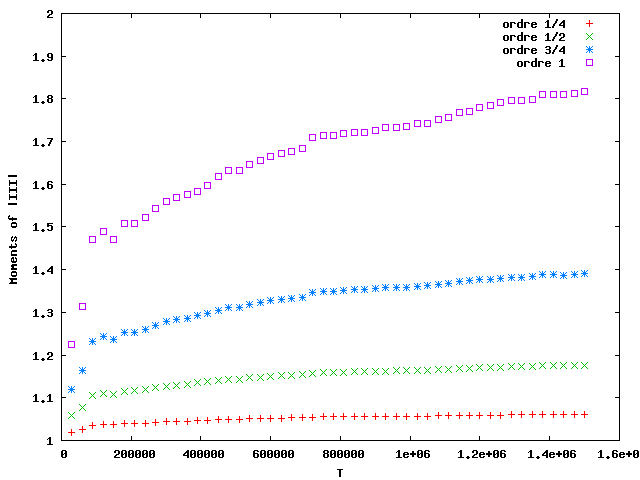}
        \vskip-10pt\caption{\label{14_sha} Moments of order
          $\frac{1}{4}$, $\frac{1}{2}$, $\frac{3}{4}$ and $1$ of the
          order of the Tate-Shafarevich groups of the twists of
          $14a1$. The heuristics suggest that the moments of order
          $k<1$ tend to a constant (depending on $k$) whereas the
          moment of order $1$ should tend to infinity.}
      \end{minipage}  
    }
  \end{small} 
\end{figure}

\begin{figure}[htbp]
  \begin{small}
    \centering
    \hspace*{-1.5cm}\mbox{
      \subfigure{
        \includegraphics[width=6cm]{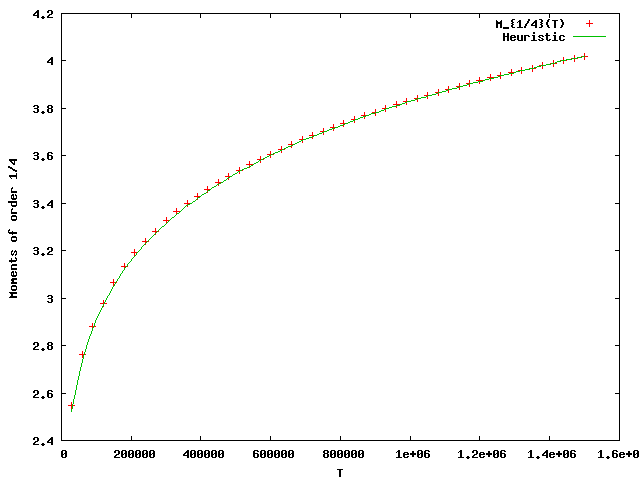}
      }
      \subfigure{
        \includegraphics[width=6cm]{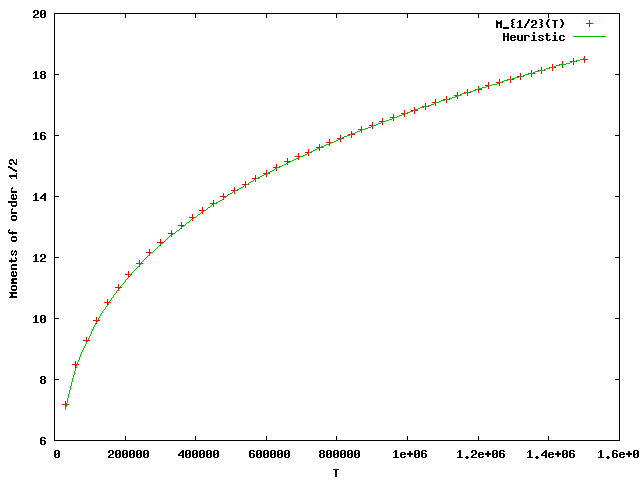}
      }
    }
    \hspace*{-1.5cm}\mbox{
      \subfigure{
        \includegraphics[width=6cm]{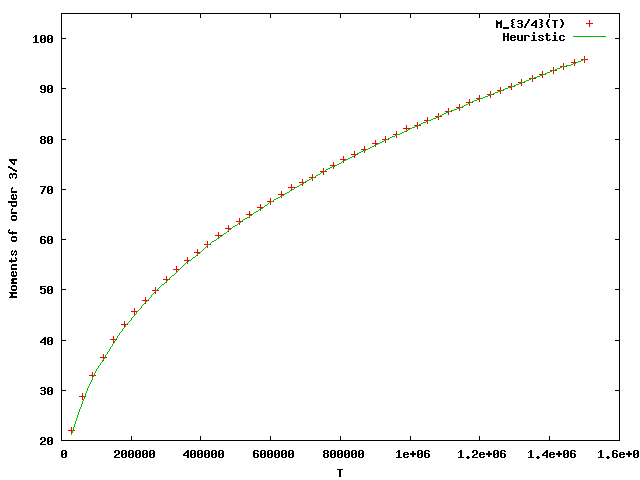}
      }
      \subfigure{
        \includegraphics[width=6cm]{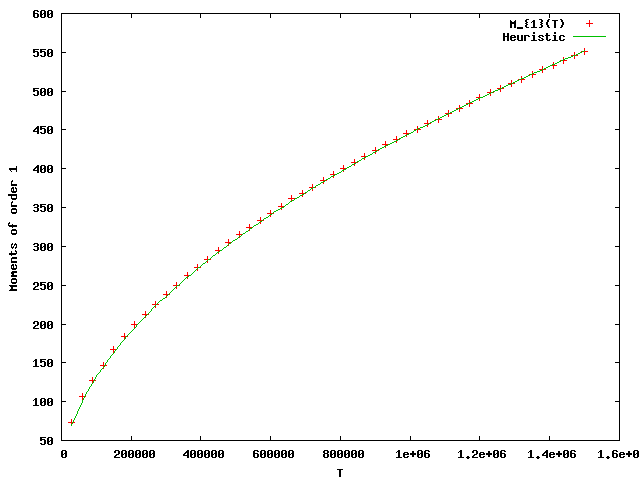}
      }
    }
    \vskip-10pt\caption{\label{14_moments} Moments of order
      $\frac{1}{4}$, $\frac{1}{2}$, $\frac{3}{4}$ and $1$ of the
      regulators of the twists of $14a1$ and the functions given by
      the heuristics.} 
  \end{small}
\end{figure}

We have $E(\Q)_{{\rm tors}} \simeq \Z/3\Z$, and there is no curve in
the isogeny class having full rational 2-torsion. Hence the heuristics
predict that
$$
M_k(T) \sim A_k T^{k/2}
\log(T)^{\frac{k(k+1)}{2}+\frac{1}{2}(4^{-k}+2^{-k}) -1} 
$$
for some constant $A_k$. We computed $A_k$ numerically to fit the data
(values found: $A_{1/4} \approx 0.82$, $A_{1/2} \approx 0.56$,
$A_{3/4} \approx 0.33$, $A_{1} \approx 0.17$) and we plot the graph of
the function given by the heuristics and the points $(T,M_k(T))$ for
$T=1,2,\dots 150 \times 3 \cdot 10^4, $ and $k=1/4$, $1/2$, $3/4$ and
$1$ (see Figure~\ref{14_moments}).

\begin{figure}[htbp]
  \begin{small}
    \hspace*{-1.5cm}\begin{minipage}[t]{8cm}
      \includegraphics[width=6cm]{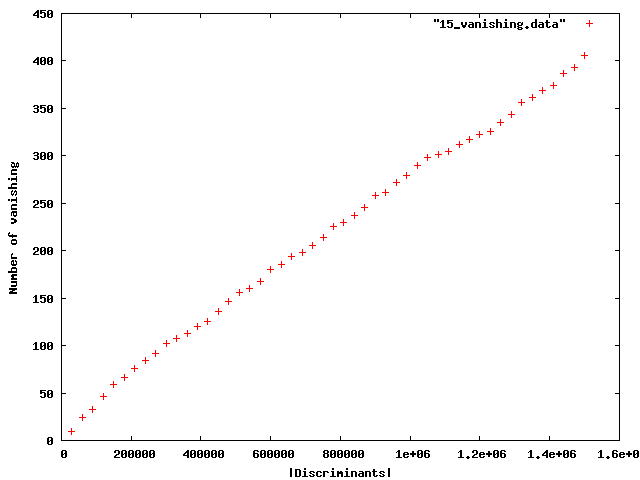}
      \vskip-10pt\caption{\label{15_extra_vanishing} Number of
        extra-vanishing for $L'(E_d,1)$ for $E=15a1$.} 
    \end{minipage}
    \hskip.3cm
    \begin{minipage}[t]{8cm}
      \includegraphics[width=6cm]{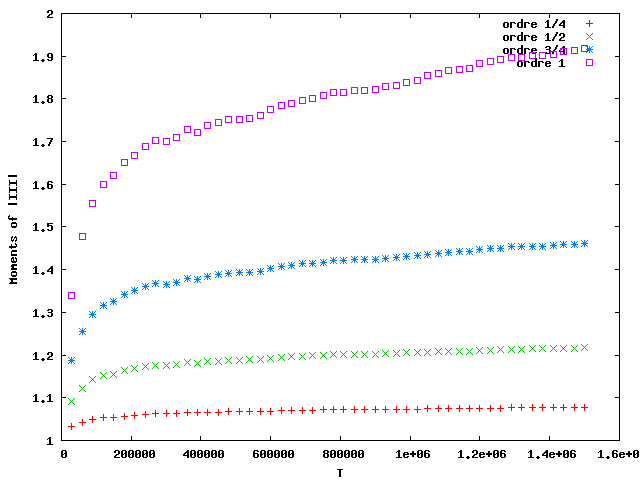}
      \vskip-10pt\caption{\label{15_sha} Moments of order
        $\frac{1}{4}$, $\frac{1}{2}$, $\frac{3}{4}$ and $1$ of the
        order of the Tate-Shafarevich groups of the twists of
        15a1. The heuristics suggest that the moments of order $k<1$
        tend to a constant (depending on $k$) whereas the moment of
        order $1$ should tend to $\infty$.}
    \end{minipage}
  \end{small}
\end{figure}

\begin{figure}[htbp]
  \begin{small}
    \centering
    \hspace*{-1.5cm}\mbox{
      \subfigure{
        \includegraphics[width=6cm]{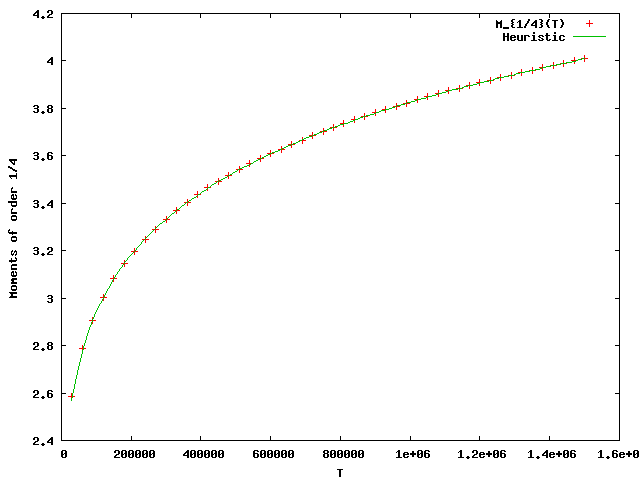}
      }
      \subfigure{
        \includegraphics[width=6cm]{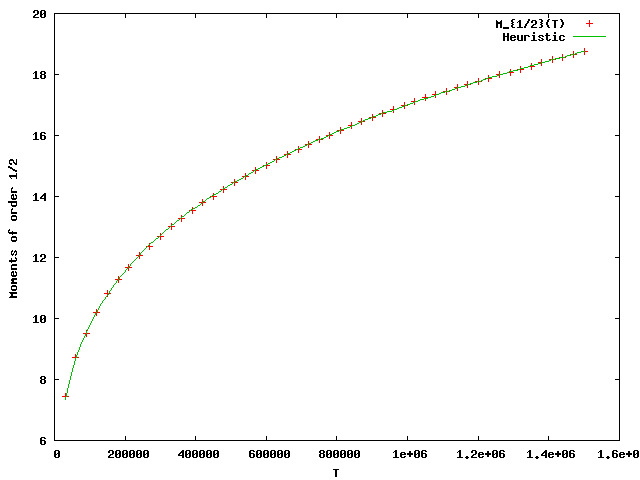}
      }
    }
    \hspace*{-1.5cm}\mbox{
      \subfigure{
        \includegraphics[width=6cm]{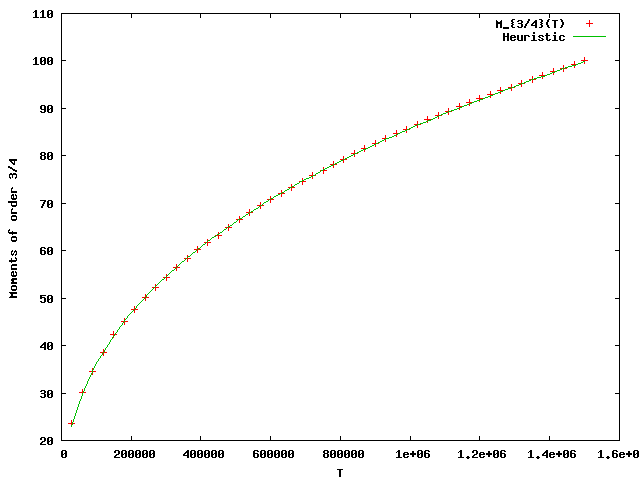}
      }
      \subfigure{
        \includegraphics[width=6cm]{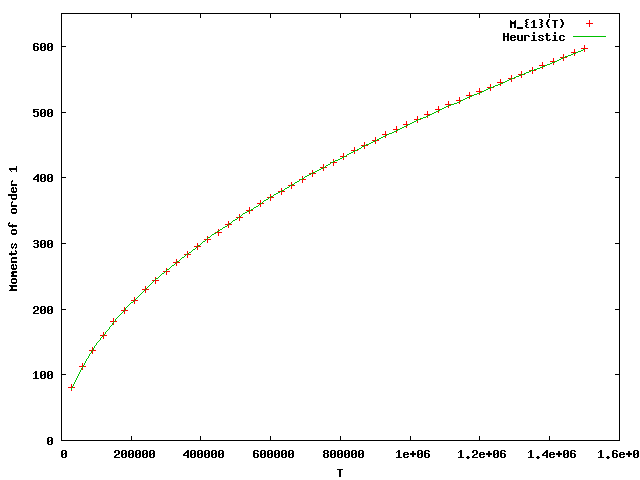}
      }
    }  
    \vskip-10pt\caption{\label{15_moments} Moments of order
      $\frac{1}{4}$, $\frac{1}{2}$, $\frac{3}{4}$ and $1$ for the
      regulators of the twists of $15a1$ and the functions given by
      the heuristics.} 
  \end{small}
\end{figure}

\subsection{The curve 15a1}

The curve $E$ is defined by $E \; : \; y^2+ xy +y =x^3+x^2-10x
-10$. It has conductor $N = 15$ and rank $0$ over $\Q$. We have $w_{3}
= w_{5} = +1$.

\begin{itemize}
\item Number of discriminants: 71254.
\item Largest generator: $\approx 19352$ (for $d=-1297619$).
\item Number of extra-vanishing: 406.
\end{itemize}

We have $E(\Q)_{{\rm tors}} \simeq \Z/4\Z \times \Z/2\Z$, hence it
has full $2$-torsion. The heuristics predict that
$$
M_k(T) \sim A_k T^{k/2} \log(T)^{k(k+1)/2+ 4^{-k} -1}
$$
for some constant $A_k$. We computed $A_k$ numerically to fit the data
(values found: $A_{1/4} \approx 0.97$, $A_{1/2} \approx 0.75$,
$A_{3/4} \approx 0.47$, $A_{1} \approx 0.25$) and we plot the graph of
the function given by the heuristic and the points $(T,M_k(T))$ for
$T=1,2,\dots 150 \times 3 \cdot 10^4, $ and $k=1/4$, $1/2$, $3/4$ and
$1$ (see Figure~\ref{15_moments}).


\begin{thebibliography}{}

\bibitem[An-Bu-Fr]{antoniadis-frey} J.~A.~Antoniadis, M.~Bungert and
  G.~Frey, {\it Properties of twists of elliptic curves}, J. Reine
  Angew. Math.  {\bf 405} (1990), 1--28.

\bibitem[Coh1]{cohen0} H.~Cohen, {\it A course in Computational
    Algebraic Number Theory}, Graduate texts in Math. {\bf 138},
  Springer-Verlag, New-York (1993).

\bibitem[Coh2]{cohen} H.~Cohen, {\it Diophantine equations, $p$-adic
    Numbers and $L$-functions}, Springer-Verlag - Graduate Texts in
  Mathematics {\bf 239} and {\bf 240}.

\bibitem[CKRS]{CKRS} J.~B.~Conrey, J.~P.~Keating, M.~O.~Rubinstein and
  N.~C.~Snaith, {\it On the frequency of vanishing of quadratic twists
    of modular $L$-functions}, Number theory for the millennium, I
  (Urbana, IL, 2000), 301--315, A. K. Peters, Natick, MA, 2002.

\bibitem[CFKRS]{CFKRS} J.~B.~Conrey, D.~W.~Farmer J.~P.~Keating,
  M.~O.~Rubinstein and N.~C.~Snaith, {\it Integral moments of
    $L$-functions}, Proc. London Math. Soc. (3) {\bf 91} (2005),
  no. 1, 33--104.

\bibitem[CRSW]{CRSW} J.~B.~Conrey, M.~O.~Rubinstein, N.~C.~Snaith and
  M.~Watkins, {\it Discretisation for odd quadratic twists}, in Ranks
  of elliptic curves and random matrix theory, ed. J.~B.~Conrey,
  D.~W.~Farmer, F.~Mezzadri and N.~C.~Snaith, London Mathematical
  Society, Lecture notes series {\bf 341}, 201--214. 

\bibitem[De1]{delaunay} C.~Delaunay, {\it Heuristics on class groups
    and on Tate-Shafarevitch groups}, in Ranks of elliptic curves and
  random matrix theory, ed. J.~B.~Conrey, D.~W.~Farmer, F.~Mezzadri
  and N.~C.~Snaith, London Mathematical Society, Lecture notes series
  {\bf 341}, 323--340.

\bibitem[De2]{delaunay2} C.~Delaunay, {\it Moments of the Orders of
    Tate-Shafarevich groups}, International Journal of Number Theory,
  {\bf 1} (2005), no. 2, 243--264.

\bibitem[De-Du]{del-duq} C.~Delaunay and S.~Duquesne, {\it Numerical
    Investigations Related to the Derivatives of the $L$-series of
    Certain Elliptic Curves}, Exp. Math. {\bf 12} (2003), no. 3,
  311--317.

\bibitem[Elk]{elkies} N.~Elkies, {\it Heegner point computations}, 
  Algorithmic number theory (Ithaca, NY, 1994), 122--133, Lecture
  Notes in Comput. Sci., 877, Springer, Berlin, 1994.

\bibitem[Hay]{hayashi} Y.~Hayashi, {\it The Rankin's $L$-function and
    Heegner points for general discriminants}, Proc. Japan
  Acad. Ser. A Math. Sci. {\bf 71} (1995), no. 2, 30--32.

\bibitem[K-S]{keating-snaith} J.~P.~Keating and N.~C.~Snaith, {\it
    Random matrix theory and $L$-functions at $s=1/2$},
  Comm. Math. Phys. {\bf 214} (2000), 91--110.

\bibitem[Kri]{krir} M.~Krir, {\it \`A propos de la conjecture de Lang
    sur la minoration de la hauteur de N\'eron-Tate pour les courbes
    elliptiques sur $\Q$}, Acta Arithmetica, {\bf C} (2001), no. 1,
  1--16.

\bibitem[Gro-Zag]{gross-zagier} B.~Gross and D.~Zagier, {\it Heegner
    points and derivatives of L-series}, Invent.~Math. {\bf 84},
  (1986), 225--320.

\bibitem[PARI]{pari} C.~Batut, K.~Belabas, D.~Bernardi, H.~Cohen and
  M.~Olivier, PARI/GP System, available at
  \texttt{http://pari.math.u-bordeaux.fr}


\bibitem[Qua]{quattrini} P.~Quattrini, {\it On the distribution of
    analytic $\sqrt{\vert \TSg \vert }$ values on quadratic twists of
    elliptic curves}, Experiment. Math.  {\bf 15} (2006), no. 3,
  355--365.

\bibitem[Ri-Vi]{ricotta-vidick} G.~Ricotta and T.~Vidick, {\it Hauteur
    Asymptotique des points de Heegner}, preprint.

\bibitem[Rub]{rubinstein} M.~Rubinstein, {\it Numerical data}, available at \\
  \texttt{http://www.math.uwaterloo.ca/$\sim$mrubinst/}

\bibitem[Sil]{silverman} J.~H.~Silverman {\it The Arithmetic of
    Elliptic Curves}, Graduate text in Math. {\bf 106},
  Springer-Verlag, New-York (1986).

\bibitem[Sna]{snaith} N.~C.~Snaith, {\it Derivatives of random matrix
    characteristic polynomials with applications to elliptic curves},
  J. Phys. A 38 (2005), {\bf 48}, 10345--10360.

\bibitem[Wat]{watkins} M.~Watkins, {\it Extra rank for odd parity
    twists}, available at \\
  \texttt{http://www.maths.bris.ac.uk/~mamjw/papers/papers.html} 

\end{thebibliography}
\end{document}